\documentclass{article}
\usepackage[margin=1.25in]{geometry}
\usepackage{amsmath}
\usepackage{amssymb}
\usepackage{amsthm}
\usepackage[all]{xy}
\usepackage{hyperref}
\usepackage{mathtools}
\usepackage{enumerate}
\usepackage{dsfont}
\usepackage{mathrsfs}

\DeclarePairedDelimiter\fat{\lVert}{\rVert}
\DeclareMathOperator{\Aut}{Aut}
\newcommand{\Z}{\mathbb{Z}}
\newcommand{\R}{\mathbb{R}}

\newcommand{\RP}{\mathbb{RP}}
\newcommand{\colim}{\mathop{\textnormal{colim}}}
\newcommand{\cm}{{\mathfrak{c}}}
\newcommand{\co}{\colon\thinspace}

\newtheorem{theorem}{Theorem}
\newtheorem{lemma}{Lemma}
\newtheorem{proposition}[lemma]{Proposition}
\newtheorem{corollary}[lemma]{Corollary}
\newtheorem{question}[lemma]{Question}
\theoremstyle{definition}
\newtheorem{definition}[lemma]{Definition}

\newtheorem{remark}[lemma]{Remark}

\title{Higher generation by abelian subgroups in Lie groups}
\author{Omar Antol\'{i}n-Camarena, Simon Gritschacher and Bernardo Villarreal}
\date{\today}

\begin{document}

\maketitle

\begin{abstract}
To a compact Lie group $G$ one can associate a space $E(2,G)$ akin to the poset of cosets of abelian subgroups of a discrete group. The space $E(2,G)$ was introduced by Adem, F. Cohen and Torres-Giese, and subsequently studied by Adem and G{\'o}mez, and other authors. In this short note, we prove that $G$ is abelian if and only if $\pi_i(E(2,G))=0$ for $i=1,2,4$. This is a Lie group analogue of the fact that the poset of cosets of abelian subgroups of a discrete group is simply--connected if and only if the group is abelian. 
\end{abstract}

\section{Introduction}

Suppose that $G$ is a discrete group and $\mathcal{F}$ is a family of subgroups of $G$. One can associate to $\mathcal{F}$ a simplicial complex $\mathscr{C}(\mathcal{F},G)$ whose $n$-simplices are the chains of cosets $g_0 A_0\subset g_1 A_1\subset \cdots \subset g_n A_n$ where $g_i\in G$ and $A_i\in \mathcal{F}$ for all $0\leqslant i\leqslant n$. It is the order complex of what is commonly called the coset poset associated to the pair $(\mathcal{F},G)$. A natural question to ask is how the topological properties of $\mathscr{C}(\mathcal{F},G)$ are related to the algebraic properties of $\mathcal{F}$ and $G$. This question was studied by Abels and Holz \cite{AbelsHolz} in some generality, in particular with regards to the higher connectivity of $\mathscr{C}(\mathcal{F},G)$. For example, $\mathscr{C}(\mathcal{F},G)$ is connected if and only if $\mathcal{F}$ covers $G$, and $\mathscr{C}(\mathcal{F},G)$ is simply connected if and only if $G$ is isomorphic to the amalgamation of all $A\in \mathcal{F}$ along their intersections. In their terminology, $\mathcal{F}$ is \textit{$n$-generating} if $\pi_i(\mathscr{C}(\mathcal{F},G))=0$ for $i\leqslant n-1$.

A simple situation arises if one considers the family $\mathcal{A}$ of all abelian subgroups of $G$. Then $\mathscr{C}(\mathcal{A},G)$ is connected, and it is easy to show that $\mathscr{C}(\mathcal{A},G)$ is simply connected if and only if $G$ is abelian, see \cite[Proposition 4.1]{OkayCAG}. When this is the case, $\mathscr{C}(\mathcal{A},G)$ is contractible. On the other hand, it may be surprising that this statement has a direct analogue in the world of Lie groups. It is the objective of this note to formulate and prove this analogue.

First, one has to clarify the meaning of $\mathscr{C}(\mathcal{A},G)$ when $G$ itself carries a topology. The role of the complex $\mathscr{C}(\mathcal{A},G)$ will be played by the geometric realization of a simplicial space, denoted by $E(2,G)$ or $E_{\textnormal{com}}G$ in the literature. It was introduced by Adem, Cohen and Torres-Giese \cite{Ad5} who studied basic properties of $E(2,G)$ as part of a more general construction involving families of nilpotent subgroups of $G$. For compact connected Lie groups $G$, further homological and homotopical properties of $E(2,G)$ were described by Adem and G{\'o}mez \cite{Ad1}. In particular, $E(2,G)$ can be related to the coset spaces $G/A$ for closed abelian subgroups $A\subseteq G$, but the relationship is much more intricate than in the discrete case. When $G$ is discrete, then $E(2,G)$ is homotopy equivalent to $\mathscr{C}(\mathcal{A},G)$.

Our goal is then to establish a precise relationship between the vanishing of the homotopy groups of $E(2,G)$ and commutativity of $G$. To do this we promote the commutator map for $G$ to a simplicial map
\[
\mathfrak{c}\co E(2,G)\to B[G,G]\, ,
\]
which will play a key role in the proof of our main result.

\begin{samepage}
\begin{theorem} \label{thm:main}
For a compact Lie group $G$ the following assertions are equivalent:
\begin{enumerate}[(1)]
\item\label{item:G-ab} $G$ is abelian
\item\label{item:EcomG-contractible} $E(2,G)$ is contractible
\item\label{item:cm-null} $\mathfrak{c}$ is null-homotopic
\item\label{item:Ecomg-highly-conn} $\pi_i(E(2,G))=0$ for $i=1,2,4$.
\end{enumerate}
\end{theorem}
\end{samepage}

There are two situations in which a stronger statement can be made than that of Theorem \ref{thm:main}, both of which are treated implicitly in our proof. Firstly, if $G$ is an arbitrary discrete group, Proposition \ref{prop:cGLie} will show that the statement of Theorem \ref{thm:main} remains valid if (\ref{item:Ecomg-highly-conn}) is replaced by $\pi_1(E(2,G))=0$. For discrete groups the results of \cite[Section~I]{AbelsHolz} imply that $E(2,G)$ is homotopy equivalent to $\mathscr{C}(\mathcal{A},G)$. In this situation we obtain a new proof of the fact that $\mathscr{C}(\mathcal{A},G)$ is simply--connected if and only if $G$ is abelian, and Theorem \ref{thm:main} may be viewed as a Lie group analogue thereof. Secondly, if $G$ is a compact Lie group with abelian identity component, then Theorem \ref{thm:main} remains valid if (\ref{item:Ecomg-highly-conn}) is replaced by \textit{$E(2,G)$ is $2$-connected}. This is Proposition \ref{prop:G0abelian2conn}.

It should be mentioned that the equivalence (\ref{item:G-ab}) $\Longleftrightarrow$ (\ref{item:EcomG-contractible}) has a precursor in the work of Adem and G{\'o}mez \cite{Ad1} which concerns a variant of $E(2,G)$ denoted $E(2,G)_{\mathds{1}}$. In general, the space of $n$-simplices of $E(2,G)$ is not connected, and $E(2,G)_{\mathds{1}}$ is obtained by restricting to the basepoint component in each simplicial degree. It is proved in \cite[Corollary 7.5]{Ad1} that for connected $G$, $E(2,G)_{\mathds{1}}$ is rationally acyclic if and only if $E(2,G)_{\mathds{1}}$ is contractible if and only if $G$ is abelian. This statement fails to hold when $G$ is disconnected (it fails for every non-abelian discrete group, for instance). In this case, one must consider $E(2,G)$ instead. While their proof relies on a well known description of the rational cohomology of spaces of commuting elements in Lie groups, we obtain our result by a rather different approach -- more homotopical than homological.

Finally, it is worth mentioning that if not contractible, the spaces $E(2,G)$ have an interesting yet difficult to understand homotopy type. For example, by \cite{lowdim} we have
\[
E(2,SU(2))\simeq S^4\vee \Sigma^4\RP^2 \quad\textnormal{and} \quad E(2,O(2))\simeq S^2\vee S^2\vee S^3\,,
\]
while for $SU=\colim_{n\to \infty} SU(n)$ it was shown in \cite[Theorem 3.4]{G1} that
\[
E(2,SU)\simeq BSU\times BSU\langle 6\rangle\times BSU\langle 8\rangle\times \cdots\,,
\]
where $BSU\langle 2n\rangle$ is the $(2n-1)$-connected cover of $BSU$. If $G$ is an extraspecial $p$-group whose Frattini quotient has rank $2r\geqslant 4$, then the universal cover of $E(2,G)$ is homotopy equivalent to a bouquet of $r$-dimensional spheres \cite{OkaySPCE}. If $G$ is a transitively commutative group, then $E(2,G)$ is homotopy equivalent to a bouquet of circles by \cite[Proposition 8.8]{Ad5}. Other interesting properties of $E(2,G)$ are proved in \cite{AGLT, RV, StafaPolypro, TorresGieseHC}.

\subsection*{Acknowledgments}

SG received funding from the European Union's Horizon 2020 research and innovation programme under the Marie Sklodowska-Curie grant agreement No. 846448. BV acknowledges support from the European Research Council (ERC) under the European Union's Horizon 2020 research and innovation programme (grant agreement No. 682922). BV was also partially supported by Mexico's CONACYT `Programa de Becas de Posgrado y apoyos a la calidad, en la Modalidad de Estancia Posdoctoral en el Extranjero.' This project was also supported by the Danish National Research Foundation through the Copenhagen Centre for Geometry and Topology (DNRF151).

\section{The simplicial space of affinely commuting elements} \label{sec:affinelycommuting}

Let $G$ be a group. We begin by recalling the simplicial bar construction for $G$, since it will form the basis for our constructions in the current and the following sections. The simplicial bar construction for the classifying space of $G$ is the simplicial space $B_\bullet G$ with $n$-simplices
\[
B_n(G):=G^n\,,
\]
face maps
\[
\partial_i\co B_{n}(G)\to B_{n-1}(G)\,, \quad \partial_i(g_1,\dots,g_n):=\begin{cases} (g_1,\dots, g_ig_{i+1},\dots,g_n) & \textnormal{if } 0<i<n \\ (g_1,\dots,g_n) & \textnormal{if }i=0 \\ (g_1,\dots,g_{n-1}) & \textnormal{if } i=n, \end{cases}
\]
and degeneracy maps $s_i\co B_{n}(G)\to B_{n+1}(G)$ given by inserting the identity element $1\in G$ in the $(i+1)$-st position. Similarly, one defines a simplicial space $E_{\bullet} G$ with $n$-simplices
\[
E_n(G):=G^{n+1}\,,
\]
face maps
\[
\partial_i\co E_n(G)\to E_{n-1}(G)\,, \quad \partial_i(g_0,\dots,g_n):=(g_0,\dots,\hat{g}_i,\dots,g_n)\,,
\]
and degeneracy maps $s_i\co E_n(G)\to E_{n+1}(G)$ given by duplicating the $i$-th coordinate. For every $n\geqslant 0$ the group $G$ acts on $E_n(G)$ diagonally by left translation, and this extends to an action on the simplicial space $E_\bullet G$. The quotient map $p\co E_\bullet G\to E_\bullet G/G\cong B_\bullet G$ can be identified with the simplicial map  given on $n$-simplices by
\[
E_{n}(G)\to B_n(G)\,, \quad (g_0,\dots,g_n)\mapsto (g_0^{-1}g_1,\dots,g_{n-1}^{-1}g_n)\, .
\]

Now let us assume for a moment that $G$ is a discrete group. Let $\mathcal{A}$ be the set of abelian subgroups of $G$ partially ordered by inclusion. We may form the union $\bigcup_{A\in \mathcal{A}}B_\bullet A$ inside $B_\bullet G$ and consider the pullback of simplicial sets
\[
\xymatrix{
E_\bullet(2,G)\ar[d]\ar[r] & E_\bullet G\ar[d]^-{p} \\
\bigcup_{A\in \mathcal{A}}B_\bullet A\ar[r] & B_\bullet G
}
\]
The pullback, which we denote by $E_\bullet(2,G)$, can be identified with the simplicial subset of $E_\bullet G$ consisting of those simplices $(g_0,\dots,g_n)\in E_n(G)$ for which $(g_{0}^{-1}g_1,\dots,g_{n-1}^{-1}g_n)\in B_n(A)$ for some abelian subgroup $A\subseteq G$. As $EG$ is contractible, the geometric realization $E(2,G)$ is the homotopy fiber of the inclusion $\bigcup_{A\in \mathcal{A}}BA\to BG$. It is therefore a measure for how well $\bigcup_{A\in \mathcal{A}}BA$ approximates $BG$. In other words, it is a measure for the group's failure to be commutative.

By the results of \cite[Section~I]{AbelsHolz}, $E(2,G)$ is homotopy equivalent to $\mathscr{C}(\mathcal{A},G)$. The same simplicial construction, however, can be carried out for an arbitrary topological group. First, observe:

\begin{lemma} \label{lem:aff}
Let $G$ be a group. The following conditions on a finite subset $S=\{s_0,\ldots,s_n\}$ of $G$ are equivalent:
\begin{enumerate}[(1)]
\item \label{item:affcom-difference} The elements $s_0^{-1}s_1,s_1^{-1}s_{2},\ldots,s_{n-1}^{-1}s_{n}$ pairwise commute.

\item \label{item:affcom-group} The group $\langle S^{-1}S \rangle:=\langle s_i^{-1}s_j \mid s_i,s_j\in S\rangle$ is abelian.
\item \label{item:affcom-coset} The set $S$ is contained in a single left coset of some abelian subgroup of $G$.
\end{enumerate}
\end{lemma}
\begin{proof}
Condition (\ref{item:affcom-group}) follows from (\ref{item:affcom-difference}), because each generator of $\langle S^{-1}S \rangle$ can be written as a product of the elements in (\ref{item:affcom-difference}). Condition (\ref{item:affcom-group}) implies (\ref{item:affcom-coset}), because $S\subset s_0 \langle S^{-1}S \rangle$. The proof is completed by showing  $(\ref{item:affcom-coset})\Longrightarrow (\ref{item:affcom-difference})$, which is immediate.
\end{proof}

\begin{definition}
We say that a finite subset $\{g_0,\ldots,g_n\}\subset G$ is \emph{affinely commutative} if it satisfies any of the equivalent conditions listed in Lemma \ref{lem:aff}.
\end{definition}

Let $G$ be a topological group. For each $n\geqslant 0$ consider the space
\[
E_n(2,G):=\{(g_0,\ldots,g_n)\in G^{n+1}\mid \{g_0,\dots,g_n\} \text{ is affinely commutative}\}\, ,
\]
with the topology induced from $G^{n+1}$. These spaces form a sub-simplicial space of $E_\bullet G$ as it can be readily seen that if $\{g_0,\dots,g_n\}$ is affinely commutative, then so are $\{g_0,\dots,\hat{g}_i,\dots,g_n\}$ as well as $\{g_0,\dots,g_i,g_i,\dots,g_n\}$ for any $0\leqslant i\leqslant n$. We denote its geometric realization by
\[
E(2,G):=|E_\bullet(2,G)|\,.
\]

\begin{remark} \label{rem:hommodel}
The space $E(2,G)$ was studied by Adem, Cohen and Torres-Giese in \cite{Ad5}, where the construction was based on a different but isomorphic model of $E_\bullet G$. Namely, let $\mathcal{E}_\bullet G$ denote the simplicial space with $n$-simplices $G^{n+1}$, face maps $\partial_i(g_0,\dots,g_n):=(g_0,\dots,g_ig_{i+1},\dots,g_n)$ for $0\leqslant i < n$ and $\partial_n(g_0,\dots,g_n)=(g_0,g_1,\dots,g_{n-1})$, and degeneracy maps $s_i$ given by inserting the identity element $1\in G$ in the $(i+1)$-st position. Then the map $E_\bullet G\to \mathcal{E}_\bullet G$ given on $n$-simplices by
\[
(g_0,\dots,g_n)\mapsto (g_0,g_0^{-1}g_1,\dots,g_{n-1}^{-1}g_n)
\]
is an isomorphism. Under this isomorphism $E_\bullet(2,G)$ becomes the simplicial space considered in \cite{Ad5}.
\end{remark}

We will need below a description of the fundamental group of $E(2,G)$ when $G$ is discrete. In this case, $E(2,G)$ is the realization of a simplicial set and a standard presentation of its fundamental group can be given, see for example \cite[Proposition 2.7, p.~126]{AschKessarOliver}. To this end, we introduce for each $(g,h)\in G^2$ a formal variable $x_{g,h}$ and set $X:=\{x_{g,h}\mid (g,h)\in G^2\}$. Let us choose the 0-simplex $1\in G$ as the basepoint for $E(2,G)$.

\begin{lemma} \label{lem:pi1}
Let $G$ be discrete. Then, the fundamental group of $E(2,G)$ admits the presentation
\begin{equation*} \label{eq:presentation}
\pi_1(E(2,G),1)=\bigl\langle X \; \bigm| \left\{ x_{g,1}, x_{1,g} \mid g \in G \right\} \cup \left\{  x_{g,h}x_{h,k}x_{k,g} \mid \{g,h,k\}\subset G \textnormal{ is affinely commutative}\right\} \bigr\rangle.
\end{equation*}
\end{lemma}

Specifically, the generator $x_{g,h}$ is represented by the loop in $E(2,G)$ obtained by concatenating the straight paths from $e$ to $g$ to $h$ to $1$, following the 1-simplices $(1,g)$, $(g,h)$ and $(h,1)$, respectively.

\section{The commutator map} \label{sec:commutatormap}

In this section we introduce our key tool, a natural map $\cm\co E(2,G)\to B[G,G]$ whose homotopy class will inform about contractibility of $E(2,G)$. The construction of $\cm$ will be possible, because of the following simple but crucial observation.

\begin{lemma}\label{lem:afftrip}
Let $\{g,h,k\}\subset G$ be an affinely commutative set. Then $[g,h][h,k]=[g,k]$.
\end{lemma}
\begin{proof}
By hypothesis, $[g^{-1}h,h^{-1}k]=[h^{-1}g,k^{-1}h]=1$, and thus \[[g,h][h,k] = g^{-1}(h^{-1}g)(k^{-1}h)k = g^{-1}(k^{-1}h)(h^{-1}g)k = g^{-1}k^{-1}gk.\qedhere\]
\end{proof}

Lemma \ref{lem:afftrip} is precisely what is needed to verify the following.

\begin{corollary}
The maps
\begin{alignat*}{3}
\cm_n\co &&  E_n(2,G) \to &\; [G,G]^n \\
&& (g_0,\ldots,g_n)\mapsto  &\;  ([g_0,g_1],\ldots,[g_{n-1},g_n])\,,
\end{alignat*}
defined for all $n\geqslant 0$, assemble into a map of simplicial spaces $\mathfrak{c}_\bullet \co E_\bullet(2,G)\to B_\bullet [G,G]$.
\end{corollary}

\begin{definition}
Upon geometric realization $\mathfrak{c}_\bullet$ defines a map
\[
\cm\co E(2,G)\to B[G,G]\,,
\]
which we refer to as the \emph{commutator map}.
\end{definition}

The rest of this section is devoted to establishing some basic properties of $\cm$. Let
\[
c\co G\times G\to G\,,\quad(x,y)\mapsto [x,y]
\]
be the \emph{algebraic commutator map} for $G$. Note that $c$ factors through a map $\tilde{c}\co G\wedge G\to [G,G]$, since $[g,h]=1$ if either $g=1$ or $h=1$.

\begin{definition}
A topological group $G$ is called \emph{homotopy abelian} if the algebraic commutator map $c$ is null-homotopic.
\end{definition}

The following proposition summarizes the main features of the commutator map $\cm$ that the proof of Theorem \ref{thm:main} will rely on.

\begin{proposition} \label{prop:cGLie}
Let $G$ be either a discrete group or a compact Lie group, let $\cm\co E(2,G)\to B[G,G]$ be the commutator map, and let \[\cm_\ast\co \pi_1(E(2,G))\to \pi_1(B[G,G])\] be the map induced by $\cm$ on fundamental groups.
\begin{enumerate}[(1)]
\item\label{item:on-gen} If $G$ is discrete, then $\cm_\ast$ satisfies $\cm_\ast(x_{g,h})=[g,h]$ for all $x_{g,h}\in \pi_1(E(2,G))$.
\item\label{item:cm-null-or-EcomG-2d+1-conn} If either $\cm$ is null-homotopic or $E(2,G)$ is $(2\dim G+1)$-connected, then $G$ is homotopy abelian.
\item\label{item:cm-0-on-pi1} The map $\cm_\ast$ is surjective, and it is trivial if and only if $[G,G]$ is a connected Lie group.
\end{enumerate}
\end{proposition}
\begin{proof}
First, assume that $G$ is discrete. As pointed out at the end of Section \ref{sec:affinelycommuting}, the generator $x_{g,h}\in \pi_1(E(2,G))$ is represented by the path obtained by concatenating the 1-simplices $(1,g)$, $(g,h)$ and $(h,1)$. Statement (\ref{item:on-gen}) follows, because $\mathfrak{c}$ takes these 1-simplices to $[1,g]=1$, $[g,h]$ and $[h,1]=1$ in $[G,G]$, respectively. 

Next we prove (\ref{item:cm-null-or-EcomG-2d+1-conn}). Since $G$ is either discrete or a Lie group, the simplicial space $E_{\bullet}(2,G)$ is proper (cf. \cite[Appendix]{Ad1}), hence the fat and thin realizations are naturally homotopy equivalent: $\fat{E_{\bullet}(2,G)}\simeq E(2,G)$. If $X$ is the geometric realization of a semi-simplicial space, we denote by $F_kX$ the $k$-th term in the skeletal filtration of $X$. Then, $F_1\fat{E_\bullet G} =F_1\fat{E_\bullet(2,G)}\cong G\ast G$ is the topological join, and $F_1\fat{B_\bullet G} \cong SG$ is the unreduced suspension. There is a commutative diagram
\begin{equation}\label{cdiagram}
  \begin{gathered}
\xymatrix{
G\ast G\ar[r]^{c'}\ar[d]&S [G,G] \ar[d] \\
\fat{E_\bullet(2,G)}\ar[r]^{\cm}&\fat{B_\bullet [G,G]}
 }
\end{gathered}
\end{equation}
where $c'([t,g,h])=[t,[g,h]]$ for $t\in [0,1]$ and $g,h\in G$. Up to homotopy, $c'$ can be identified with the map $\Sigma \tilde{c}\co \Sigma G\wedge G\to \Sigma [G,G]$. If $\cm$ is null-homotopic, diagram (\ref{cdiagram}) implies that the composite
\[
\Sigma(G\wedge G)\xrightarrow{\Sigma \tilde{c}}\Sigma [G,G]\to B[G,G]
\]
is null-homotopic as well. We also get that this composite is null-homotopic if $\fat{E_\bullet(2,G)}$ is $(2 \dim G + 1)$-connected, since then the map $G\ast G \to \fat{E_\bullet(2,G)}$ appearing in the diagram is null-homotopic by a standard obstruction theory argument. As a map between path connected spaces is null-homotopic if and only if it is based null-homotopic, the adjoint map
\[
G\wedge G\xrightarrow{\tilde{c}}[G,G]\xrightarrow{\simeq}\Omega B[G,G]\,,
\]
and hence $\tilde{c}$, are null-homotopic. Since the algebraic commutator map $c\co G\times G\to G$ factors through $\tilde{c}$, it is null-homotopic as well. This finishes the proof of (\ref{item:cm-null-or-EcomG-2d+1-conn}).

Now we prove (\ref{item:cm-0-on-pi1}). If $G$ is discrete it follows directly from statement (\ref{item:on-gen}). Assume $G$ is a compact Lie group and let $G^{\delta}$ denote $G$ equipped with the discrete topology. Let $d\co G^{\delta}\to G$ be the canonical map. The commutator map $\mathfrak{c}$ for $G$ and the commutator map $\mathfrak{c}^{\delta}$ for $G^{\delta}$ are related by a commutative diagram
\[
\xymatrix{
E(2,G^{\delta}) \ar[d]^-{E(2,d)} \ar[r]^{\mathfrak{c}^{\delta}} & B[G^\delta,G^{\delta}] \ar[d]^-{Bd} \\
E(2,G) \ar[r]^{\mathfrak{c}} & B[G,G].
}
\]
Recall that for Lie groups, the commutator subgroup $[G,G]$ is defined to be the closure of the algebraic commutator subgroup. But the commutator subgroup of a compact Lie group is always closed (see \cite[Theorem~6.11]{HM13}), so $[G,G]^\delta = [G^\delta,G^\delta]$.

The diagram induces a commutative diagram on fundamental groups. Now consider the composite homomorphism
\begin{equation} \label{eq:compositesurj}
\pi_1(E(2,G^\delta))\xrightarrow{\mathfrak{c}^{\delta}_{\ast}} [G^{\delta},G^{\delta}]\xrightarrow{\pi_1(Bd)} \pi_1(B[G,G])
\end{equation}
obtained by going through the top right corner of the diagram. Under the isomorphism $\pi_1(B[G,G])\cong \pi_0([G,G])$ and the identification of $[G^{\delta},G^{\delta}]$ with $[G,G]^{\delta}$ the map $\pi_1(Bd)$ corresponds to the canonical surjection $[G,G]^{\delta} \to \pi_0([G,G])$. Moreover, by part (\ref{item:on-gen}) the map $\mathfrak{c}^{\delta}_{\ast}$ is surjective. Together this implies that (\ref{eq:compositesurj}) is surjective, and by commutativity of the diagram $\mathfrak{c}_\ast$ must be surjective, as well. In particular, if $\mathfrak{c}_\ast$ is trivial, then $\pi_0([G,G])=1$. Conversely, if $[G,G]$ is path connected, then $B[G,G]$ is simply connected, hence $\mathfrak{c}_\ast$ is trivial.
\end{proof}

\begin{remark} \label{rem:homotopyabeliandoesnotimplycontractible}
Perhaps surprisingly, there exist homotopy abelian compact Lie groups $G$ for which $\cm$ is not null-homotopic. Hence, the converse of part (\ref{item:cm-null-or-EcomG-2d+1-conn}) of Proposition \ref{prop:cGLie} fails to hold. An example illustrating this is the central extension
\[
1\to S^1\to (S^1\times Q_8)/\Z/2 \to \Z/2\times \Z/2 \to 1\, ,
\]
where $Q_8$ is the quaternion group of order eight. The quotient $G=(S^1\times Q_8)/\Z/2$ is taken over the central subgroup $\Z/2=\langle(-1,-1)\rangle\subset S^1\times Q_8$. It is indeed homotopy abelian; the commutator subgroup is $[G,G]=\{[(1,1)],[(-1,1)]\}\cong \Z/2$, which is a discrete subgroup of the path-connected group $S^1$, thus making the algebraic commutator map null-homotopic. But by part (\ref{item:cm-0-on-pi1}) of Proposition \ref{prop:cGLie}, $\cm$ cannot be trivial on fundamental groups, since $[G,G]$ is not connected.
\end{remark}

\begin{remark} \label{rem:commutatorinverse}
Let $j\co B[G,G]\to BG$ be the map induced by the inclusion $[G,G]\subseteq G$. There is another description, up to homotopy, of the composition $j\mathfrak{c}\co E(2,G)\to BG$. We shall not need it to prove our main theorem; but it seems worth mentioning, because it is not obvious from the definition. Let $C_n(G)\subseteq G^n$ denote the subspace of $n$-tuples of commuting elements in $G$. Then $C_\bullet(G)\subseteq B_\bullet G$ is a sub-simplicial space, whose realization we denote by $B(2,G)$. The composite map
\[
E(2,G)\subseteq EG \xrightarrow{p} BG
\]
factors through the inclusion $i\co B(2,G)\to BG$. By abuse of notation, we write $p\co E(2,G)\to B(2,G)$ for the projection. Note that there is an automorphism $\phi^{-1}\co B(2,G)\to B(2,G)$ induced by the map $G\to G$, $g\mapsto g^{-1}$. We claim that the diagram
\[
\xymatrix{
E(2,G) \ar[r]^-{p} \ar[d]^-{\mathfrak{c}} & B(2,G) \ar[r]^-{\phi^{-1}} & B(2,G) \ar[d]^-{i} \\
B[G,G] \ar[rr]^-{j} & & BG
}
\]
commutes up to homotopy. Indeed, it is tedious but straightforward to verify that the collection of maps $\{h_i\}_{0\leqslant i\leqslant n}$ defined by
\begin{alignat*}{1}
h_i\co E_{n}(2,G) & \to B_{n+1}(G)=G^{n+1} \\
(g_0,\dots,g_n) & \mapsto ([g_0,g_1],\dots, [g_{i-1},g_i],g_i^{-1},g_{i+1}^{-1}g_i,\dots,g_{n}^{-1}g_{n-1})
\end{alignat*}
is a simplicial homotopy between $j\mathfrak{c}$ and $i\phi^{-1}p$ in the sense of \cite[Definition 9.1]{M72}.
\end{remark}

\section{The proof of Theorem \ref{thm:main}} \label{sec:main}

The proof of Theorem \ref{thm:main} will require a couple of propositions, the first of which is a characterization of homotopy abelian compact Lie groups.

\begin{proposition} \label{prop:homotopyabelian}
Let $G$ be a compact Lie group. Then $G$ is homotopy abelian if and only if $\pi_0(G)$ is abelian and $G$ is a central extension of $\pi_0(G)$ by a torus.
\end{proposition}
\begin{proof}
Suppose that $G$ is homotopy abelian. Let $G_0\subseteq G$ be the component of the identity and let $T\subseteq G_0$ be a maximal torus. As the commutator map $c\co G\times G\to G$ is null-homotopic, it factors through $G_0$ and its restriction to $G_0$ is null-homotopic, too. It follows that $G_0$ is homotopy abelian. A result of Araki, James and Thomas \cite{AJTh} asserts that a compact, connected, homotopy abelian Lie group is abelian. Hence, $G_0=T$. Thus $G$ fits into an extension
\[
1\to T \to G \xrightarrow{p} \pi_0(G)\to 1\,.
\]
It is clear that $\pi_0(G)$ is abelian, and so it remains to show that $T$ is central.

Note that $\Aut(T)\cong \Aut(H_1(T;\Z))$ is discrete. For $g\in G$ let $\text{conj}_g\in \Aut(T)$ denote the inner automorphism $t\mapsto g^{-1}tg$. The map $g\mapsto \text{conj}_g$ must be constant on connected components and thus factors through a representation
\[
\rho\co \pi_0(G)\to \Aut(H_1(T;\Z))\,.
\]
To show that $T$ is central, it is enough to show that $\rho$ is constant. Fix $g\in G$. Let $c(-,g)\co T\to T$ denote the composition
\[
T \xrightarrow{t\mapsto (t,g)} G\times G \xrightarrow{c} T\, ,
\]
 which has image in $T=G_0$ because $c(1,g)=1$. The composite map is null-homotopic, because the commutator map is null-homotopic. On the other hand, $c(-,g)$ can be identified with the composition
\[
T \xrightarrow{t\mapsto (t^{-1},t)} T\times T \xrightarrow{id\times \text{conj}_g} T\times T \xrightarrow{\cdot } T\,,
\]
where the last map is multiplication in $T$. As this map is null-homotopic, the induced map on $H_1(T;\Z)$ is zero. This implies that, for any $x\in H_1(T;\Z)$, we must have
\[
0=-x+\rho(p(g)) (x)\,,
\]
hence $\rho(p(g))=id$. This finishes the proof that $T$ is central.

Conversely, suppose that $G$ is a central extension of $\pi_0(G)$ by a torus $T$ and assume that $\pi_0(G)$ is abelian. The central extension is classified by a 2-cocycle $\omega\co \pi_0(G)\times \pi_0(G)\to T$. As an abstract group, $G$ is isomorphic to $T\times \pi_0(G)$ with group law
\[
(t,x) (s,y) := (ts\omega(x,y),xy)\,,
\]
see \cite[Remark 18.1.14]{HN12}. A short computation shows that, when $\pi_0(G)$ is abelian, the commutator of any two elements $(t,x)$ and $(s,y)$ of $T\times \pi_0(G)$ reads
\[
[(t,x),(s,y)]=(\omega(x,y)\omega(y,x)^{-1},0)\, .
\]
Thus, the commutator map $c\co G\times G\to G$ factors through $\pi_0(G)\times \pi_0(G)$ and has image in $T$. Since $T$ is connected, $c$ is null-homotopic.
\end{proof}

\begin{remark}
The central extension $(S^1\times Q_8)/\Z/2$ described in Remark \ref{rem:homotopyabeliandoesnotimplycontractible} is an example of a homotopy abelian compact Lie group which is not abelian. It also illustrates that the theorem of Araki, James and Thomas \cite{AJTh} used in the proof of Proposition \ref{prop:homotopyabelian} fails to hold for disconnected groups.
\end{remark}

Another statement that will enter into the proof of our main result is the following.

\begin{proposition}\label{prop:pi4E2G}
Let $G$ be a compact Lie group. If $\pi_4(E(2,G))=0$, then the component of the identity $G_0\subseteq G$ is abelian, hence $G$ is an extension of $\pi_0(G)$ by a torus. 
\end{proposition}

The proof of the proposition requires some preparation. Let $C_n(G)\subseteq G^n$ denote the subspace of $n$-tuples of commuting elements in $G$.

\begin{lemma}\label{lem:reducedE2G}
The realization of the sub-simplicial space $C_{\bullet+1}(G)\subseteq E_\bullet(2,G)$ is contractible.
\end{lemma}
\begin{proof}
  Implicit in the statement is the claim that $C_{\bullet+1}(G)$ is a sub-simplicial space of $E_\bullet(2,G)$. Since in $E_\bullet(2,G)$ the faces and degeneracies delete and duplicate coordinates (as they do in the simplicial model of $EG$ described in Section \ref{sec:affinelycommuting}) it is easy to check that $C_{\bullet+1}(G)$ is indeed a sub-simplicial space.

  To prove it is contractible we can straightforwardly adapt one of the usual proofs that $EG$ is contractible: the simplicial model of $EG$ can be augmented by adding a unique $(-1)$-simplex and this augmented simplicial space has an extra degeneracy given by $s_{-1}(g_0, \ldots, g_n) = (1, g_0, \ldots, g_n)$ for any $n \ge -1$. This extra degeneracy preserves $C_{\bullet+1}(G)$ and thus also shows that its geometric realization is contactible.
\end{proof}

We now define a homotopy equivalent model for $E(2,G)$ which will turn out convenient. Consider the simplicial space $\bar{E}_{\bullet}(2,G)$ with $n$-simplices
\[
\bar{E}_{n}(2,G):=E_n(2,G)/C_{n+1}(G)
\]
and simplicial structure the one induced by $E_{\bullet}(2,G)$.

As $C_{\bullet+1}(G)\to E_\bullet(2,G)$ is a levelwise cofibration of good simplicial spaces the map of realizations $|C_{\bullet+1}(G)|\to E(2,G)$ is a cofibration. By Lemma \ref{lem:reducedE2G} $|C_{\bullet+1}(G)|$ is contractible, so the map $E(2,G)\to  E(2,G)/|C_{\bullet+1}(G)|$ is a homotopy equivalence. Since geometric realization commutes with taking cofibers, the levelwise quotient maps induce a homotopy equivalence
\[
E(2,G)\xrightarrow{\simeq} \bar{E}(2,G)\, .
\]
Just like $E(2,G)$, the assignment $G\mapsto \bar{E}(2,G)$ is natural for homomorphisms of groups, and so is the equivalence $E(2,G)\simeq \bar{E}(2,G)$.

\begin{remark}
In the introduction we mentioned the space $E(2,G)_{\mathds{1}}$, which is the geometric realization of the sub-simplicial space $E_{\bullet}(2,G)_{\mathds{1}}\subseteq E_{\bullet}(2,G)$ consisting of the connected component of $(1, \ldots, 1)$ in each degree. This space also has a homotopy equivalent model obtained as above by setting $\bar{E}_n(2,G)_{\mathds{1}}:=E_n(2,G)_{\mathds{1}}/C_{n+1}(G)_{\mathds{1}}$. Indeed, the extra degeneracy used in the proof of Lemma \ref{lem:reducedE2G} preserves the sub-simplicial space $C_{\bullet+1}(G)_{\mathds{1}}$ consisting in degree $n$ of the connected component of $C_{n+1}(G)$ containing $(1,\ldots,1)$.
\end{remark}

The commutator map $\mathfrak{c}\co E(2,G)\to B[G,G]$ factors through $\bar{E}(2,G)$. To keep the notation simple we denote the resulting map $\mathfrak{c}\co \bar{E}(2,G)\to B[G,G]$ by the same letter. Observe that $\bar{E}(2,G)$ is a reduced simplicial space, and the space of $1$-simplices is $G^2/C_2(G)$. Therefore, the simplicial $1$-skeleton is $\Sigma G^2/C_2(G)$ and the commutator map restricted to the $1$-skeleton
\[
\mathfrak{c}|\co \Sigma G^2/C_2(G)\to \Sigma [G,G]
\]
is simply the suspension of the map induced by the algebraic commutator map $c\co G^2\to [G,G]\subset G$.

\begin{lemma}\label{lem:sectionloopsE2SU2}
After looping the commutator map $\mathfrak{c}\co \bar{E}(2,SU(2))\to BSU(2)$ has a section up to homotopy, and this section $s\co SU(2)\to \Omega \bar{E}(2,SU(2))$ is natural with respect to homomorphisms $f\co SU(2)\to G$ in the sense that the diagram
\[
\xymatrix{
\Omega \bar{E}(2,SU(2)) \ar[r]^-{\Omega \bar{E}(2,f)} \ar[d]^-{\Omega \mathfrak{c}} & \Omega \bar{E}(2,G) \ar[d]^-{\Omega\mathfrak{c}} \\
SU(2) \ar@/^1.0pc/@{..>}[u]^-{s} \ar[r]^-{f} & [G,G]
}
\]
with the dotted arrow filled in commutes up to homotopy.
\end{lemma}
\begin{proof}
In the diagram we have implicitly used the canonical homotopy equivalence $[G,G]\simeq \Omega B [G,G]$ adjoint to the inclusion $\Sigma [G,G]\to B[G,G]$. By adjunction it is enough to construct a map $s'\co \Sigma SU(2)\to \bar{E}(2,SU(2))$ making the following diagram commute:
\[
\xymatrix{
& \bar{E}(2,SU(2)) \ar[r]^-{\bar{E}(2,f)} \ar[d]^-{\mathfrak{c}} & \bar{E}(2,G) \ar[d]^-{\mathfrak{c}} \\
\Sigma SU(2) \ar[r]^-{\text{incl}} \ar@{..>}[ur]^-{s'} & BSU(2)\ar[r]^-{Bf} & B[G,G]
}
\] 
The desired section $s$ may then be defined as the adjunct of $s'$. As the simplicial $1$-skeleton of $\bar{E}(2,SU(2))$ is $\Sigma SU(2)^2/C_2(SU(2))$ it suffices to construct a section of the map
\[
\mathfrak{c}|\co \Sigma SU(2)^2/C_2(SU(2)) \to \Sigma SU(2)\,,
\]
and $s'$ may be defined as the composite of this section with the inclusion into $\bar{E}(2,SU(2))$.

It is shown in \cite[Section VI 1(a)]{AkbulutMcCarthy} that the restriction of the algebraic commutator map to the non-commuting pairs in $SU(2)$,
\[
c|\colon SU(2)^2-C_2(SU(2))\to SU(2)-\{1\}\,,
\]
is a locally trivial bundle with fiber $c^{-1}(-1)$. Note that $\mathfrak{c}|=\Sigma (c|)^+$, where $(c|)^+$ is the map induced by $c|$ on one-point compactifications. As $SU(2)-\{1\}$ is contractible there is a homeomorphism of the total space of the fiber bundle with $(SU(2)-\{1\})\times c^{-1}(-1)$ under which $c|$ corresponds to the projection onto the first factor. Now $c^{-1}(-1)$ is compact, since it is a closed subset of the compact space $SU(2)^2$. Thus, there is a homeomorphism
\[
[(SU(2)-\{1\})\times c^{-1}(-1)]^+ \cong SU(2)\wedge c^{-1}(-1)_+\,,
\]
where $c^{-1}(-1)_+$ denotes $c^{-1}(-1)$ with a disjoint basepoint added. Under this homeomorphism $(c|)^+$ can be identified with the map $SU(2)\wedge c^{-1}(-1)_+\to SU(2)$ induced by the projection $c^{-1}(-1)_+\to S^0$. A choice of basepoint of $c^{-1}(-1)$ gives a section $SU(2)\wedge S^0\to  SU(2)\wedge c^{-1}(-1)_+$, and its suspension yields a section for $\mathfrak{c}|$.
\end{proof}

\begin{proof}[Proof of Proposition \ref{prop:pi4E2G}]
Let $G_0$ denote the component of the identity of $G$. We must show that $G_0$ is abelian. Clearly, this follows if we can show that $[G,G]_0$ is abelian. For $[G_0,G_0]$ is a subgroup of $[G,G]_0$, and the commutator group of a connected compact Lie group is semisimple.

Thus, assume for contradiction that $[G,G]_0$ is non-abelian. It is well known that the universal cover of a compact connected Lie group $K$ decomposes as a product of simply--connected simple Lie groups $\{K_i\}_{i=1,\dots,k}$ and a copy of $\R^m$, giving $\pi_3(K)=\pi_3(K_1)\oplus\cdots\oplus\pi_3(K_k)$. For $K=[G,G]_0$ we must have $k\geqslant 1$, since $[G,G]_0$ is assumed non-abelian. In \cite[Chapter III Proposition 10.2]{BottSamelson} it is shown that in a simply--connected simple Lie group $K_i$ one can find a subgroup isomorphic to $SU(2)$ such that the inclusion $SU(2)\to K_i$ induces an isomorphism in $\pi_3(-)$. Thus we find a homomorphism $f\co SU(2) \to [G,G]_0$ such that $\pi_3(f)\co \pi_3(SU(2))\to \pi_3([G,G]_0)$ is injective. To reach a contradiction it suffices to show that the map $\pi_3(SU(2))\to \pi_3([G,G])$ obtained by composition with the inclusion $[G,G]_0\subseteq [G,G]$ is zero.

Application of $\pi_3(-)$ to the homotopy commutative diagram in Lemma \ref{lem:sectionloopsE2SU2} yields a commutative diagram
\[
\xymatrix{
\pi_3(\Omega\bar{E}(2,SU(2))) \ar[r]^-{\pi_3(f')} & \pi_3(\Omega\bar{E}(2,G)) \ar[d]^-{\pi_3(\Omega\mathfrak{c})} \\
\pi_3(SU(2)) \ar[u]^-{\pi_3(s)} \ar[r]^-{\pi_3(f)} & \pi_3([G,G])
}
\]
where $f':=\Omega \bar{E}(2,f)$. Since $E(2,G)$ is homotopy equivalent with $\bar{E}(2,G)$, we have that
\[
\pi_3(\Omega \bar{E}(2,G))\cong  \pi_4(\bar{E}(2,G))\cong \pi_4(E(2,G))\,.
\]
By assumption this group is zero, hence $\pi_3(f)=0$. 
\end{proof}

The final item needed to prove Theorem \ref{thm:main} is the following proposition.

\begin{proposition} \label{prop:G0abelian2conn}
Let $G$ be a compact Lie group and assume that the component of the identity $G_0$ is abelian. If $E(2,G)$ is $2$-connected, then $\mathfrak{c}$ is null-homotopic.
\end{proposition}
\begin{proof}
Since $\pi_1(E(2,G))=0$ by assumption, we deduce from Proposition \ref{prop:cGLie} part (\ref{item:cm-0-on-pi1}) that $[G,G]$ is connected. Then $[G,G]\subseteq G_0$, and since $[G,G]$ is also closed it is a torus. Therefore, $B[G,G]$ is an Eilenberg-MacLane space of type $K(\Z^r,2)$ for some $r\geqslant 0$, and the homotopy class of the commutator map
\[
\cm\co E(2,G)\to B[G,G]\simeq K(\Z^r,2)
\]
corresponds to a cohomology class in $H^2(E(2,G);\Z^r)$. Since $E(2,G)$ is assumed $2$-connected, we have that $H^2(E(2,G);\Z^r)=0$. Hence $\cm$ is null-homotopic, as desired. 
\end{proof}

We can now prove the main result of this paper.

\begingroup
\def\thetheorem{\ref{thm:main}}
\begin{theorem}
For a compact Lie group $G$ the following assertions are equivalent:
\begin{enumerate}[(1)]
\item\label{item:G-ab} $G$ is abelian
\item\label{item:EcomG-contractible} $E(2,G)$ is contractible
\item\label{item:cm-null} $\mathfrak{c}$ is null-homotopic
\item\label{item:Ecomg-highly-conn} $\pi_i(E(2,G))=0$ for $i=1,2,4$.
\end{enumerate}
\end{theorem}
\endgroup

\begin{proof}
Clearly, if $G$ is abelian, then $E(2,G)$ is contractible, because in this case every subset $\{g_0,\dots,g_n\}\subseteq G$ is affinely commutative, so $E(2,G)=EG$ and $EG$ is contractible. If $E(2,G)$ is contractible, then it is obvious that $\cm$ is null-homotopic, and that $\pi_i(E(2,G))=0$ for $i=1,2$ and $4$. To establish the theorem we shall prove that (\ref{item:cm-null}) $\implies$ (\ref{item:G-ab}), and (\ref{item:Ecomg-highly-conn}) $\implies$ (\ref{item:cm-null}).

Suppose that $\cm$ is null-homotopic. Then Proposition \ref{prop:cGLie} part (\ref{item:cm-null-or-EcomG-2d+1-conn}) and Proposition \ref{prop:homotopyabelian} imply that $G$ is a central extension of $\pi_0(G)$ by a torus. In addition the map $\mathfrak{c}_\ast$ of fundamental groups is trivial, hence $[G,G]$ is a connected Lie group by Proposition \ref{prop:cGLie} part (\ref{item:cm-0-on-pi1}). Then it is a subgroup of $G_0=T$, hence a torus. It will suffice to show that $[G,G]$ is finitely generated, since a torus is finitely generated only if it is the trivial group. But as pointed out in the proof of Proposition \ref{prop:homotopyabelian}, $[G,G]$ is generated by the image of the map
\begin{alignat*}{3}
&& \pi_0(G)\times \pi_0(G)  \to &\; T \\
&& (x,y) \mapsto &\; \omega(x,y)\omega(y,x)^{-1}\,,
\end{alignat*}
which is finite, since $\pi_0(G)$ is finite. This shows that (\ref{item:cm-null}) $\implies$ (\ref{item:G-ab}).

We will now show that (\ref{item:Ecomg-highly-conn}) $\implies$ (\ref{item:cm-null}). By assumption $\pi_4(E(2,G))=0$, so Proposition \ref{prop:pi4E2G} implies that the identity component $G_0\subseteq G$ is abelian. Proposition \ref{prop:G0abelian2conn} now finishes the proof.
\end{proof}

There is an intriguing relationship of $E(2,G)$ with bundle theory. In \cite{Ad1} it is explained how the $i$-th homotopy group of $E(2,G)$ can be interpreted as the set of ``transitionally commutative structures" on the trivial principal $G$-bundle over $S^i$. We refer to \cite{Ad1} for more background. If $G$ is non-abelian, then $\pi_i(E(2,G))\neq 0$ for some $i\in\{1,2,4\}$. Therefore, our theorem has the following corollary.

\begin{corollary}
Let $G$ be a non-abelian compact Lie group. Then the trivial principal $G$-bundle over at least one of $S^1$, $S^2$ or $S^4$ admits two distinct transitionally commutative structures.
\end{corollary}

\begin{remark} The compactness condition in Theorem \ref{thm:main} is necessary. For example, $SL(2,\R)$ is a homotopy abelian Lie group as it deformation retracts onto $SO(2)$, but it is not abelian. On the other hand, a result of Pettet and Suoto \cite[Corollary 1.2]{PS} implies that $E(2,SL(2,\R))\simeq E(2,SO(2))=ESO(2)$, which is contractible.
\end{remark}

\section{A potential splitting} \label{sec:final}

The results in this paper would be well explained by a splitting up to homotopy of the looped commutator map $\Omega \mathfrak{c}$, and hence a splitting of spaces
\[
\Omega E(2,G) \simeq [G,G]\times \Omega X\,,
\]
for some space $X$. Indeed, if such a splitting exists, then any of the equivalent conditions listed in Theorem \ref{thm:main} readily implies that $[G,G]=1$. Note that a connected and simply-connected compact Lie group with trivial $\pi_3$ is necessarily trivial.

The splitting exists for $G=SU(2)$ as proved in Lemma \ref{lem:sectionloopsE2SU2}. It also exists for $G=O(2)$ and for $G=SU$. For example, we proved in \cite[Theorem 1.5]{lowdim} that $E(2,O(2))\simeq \Sigma(S^1\times S^1)$. For any group $G$ there is a homotopy fiber sequence $G\ast G\to \Sigma G \to BG$ by Ganea's theorem. After looping, the unit map $G\to \Omega \Sigma G$ splits the homotopy fiber sequence, hence
\[
\Omega \Sigma G\simeq G \times \Omega (G\ast G)\,.
\]
In particular, there is a homotopy equivalence
\[
\Omega E(2,O(2)) \simeq S^1\times S^1 \times \Omega ((S^1\times S^1)\ast(S^1\times S^1))\, .
\]
Note that $S^1=SO(2)=[O(2),O(2)]$. To prove that $\Omega \mathfrak{c}\co \Omega E(2,O(2))\to SO(2)$ splits up to homotopy it is enough to show that $\Omega \mathfrak{c}$ is surjective on fundamental groups. Because the inclusion $SO(2)\to O(2)$ induces an isomorphism $\pi_2(BSO(2))\cong \pi_2(BO(2))$, one can equivalently show, using Remark \ref{rem:commutatorinverse}, that $i\phi^{-1}p\co E(2,O(2))\to BO(2)$ is surjective on $\pi_2$. Surjectivity follows from results in \cite{RV} as we will now explain. The authors construct a map $f_1\colon S^2\to B(2,O(2))$ such that $if_1$ is null-homotopic but $i\phi^{-1}f_1$ is a generator of $\pi_2(BO(2))\cong\Z$ (\cite[Proposition 3.5]{RV}). Thus we can find a lift of $f_1$ up to homotopy $\tilde{f}_1\colon S^2\to E(2,O(2))$ such that $i\phi^{-1}p\tilde{f}_1$ is a generator of $\pi_2(BO(2))$. 

We leave it to the reader to show that $\Omega \mathfrak{c}\co \Omega E(2,SU)\to SU$ has a splitting up to homotopy using \cite[Theorem 3.4]{G1} and Remark \ref{rem:commutatorinverse}.

There are too few examples known to build a firm opinion, but the results of this paper suggest that the following question warrants further study.

\begin{question}
Let $G$ be a compact Lie group. Does the commutator map
\[
\cm\co E(2,G)\to B[G,G]
\]
split up to homotopy after looping?
\end{question}

One way of establishing a splitting is by showing that the restriction $\mathfrak{c}|$ of the commutator map to the simplicial $1$-skeleton of $\bar{E}(2,G)$ has a splitting up to homotopy. This was carried out for $G=SU(2)$ in Lemma \ref{lem:sectionloopsE2SU2}. However, one can show that $\mathfrak{c}|$ splits neither for $G=O(2)$ nor for $G=SO(3)$. For example, for $G=SO(3)$ we have $H_1(SO(3);\Z) \cong \Z/2$ but one can compute that $H_1(SO(3)^2/C_2(SO(3));\Z) \cong \Z$ using \cite[Theorem 1.2]{TGS}. This motivates the following question.

\begin{question}
For which groups $G$ does the commutator map
\[
\mathfrak{c}|\co \Sigma G^2/C_2(G)\to \Sigma [G,G]
\]
split up to homotopy?
\end{question}

{\footnotesize {\sc \noindent Omar Antol\'{\i}n-Camarena\\
   Instituto de Matem\'aticas, UNAM, Mexico City, Mexico}\\
  \emph{E-mail address}:
  \href{mailto:omar@matem.unam.mx}{\texttt{omar@matem.unam.mx}}}
\medskip

{\footnotesize {\sc \noindent Simon Gritschacher\\
  Department of Mathematical Sciences, University of Copenhagen, Copenhagen, Denmark}\\
  \emph{E-mail address}:
  \href{mailto:gritschacher@math.ku.dk}{\texttt{gritschacher@math.ku.dk}}}
\medskip

{\footnotesize {\sc \noindent Bernardo Villarreal \\
  Instituto de Matem\'aticas, UNAM, Mexico City, Mexico}\\ 
  \emph{E-mail address}:
  \href{mailto:villarreal@matem.unam.mx}{\texttt{villarreal@matem.unam.mx}}}
\end{document}